\documentclass[11pt]{amsart}
\usepackage{extarrows}
\usepackage[colorlinks, citecolor=blue, dvipdfm, pagebackref]{hyperref}

\newtheorem{theorem}{Theorem}[section]
\newtheorem{lemma}[theorem]{Lemma}

\theoremstyle{definition}
\newtheorem{definition}[theorem]{Definition}
\newtheorem{question}[theorem]{Question}
\newtheorem{example}[theorem]{Example}

\newtheorem{proposition}[theorem]{Proposition}

\newtheorem{remark}[theorem]{Remark}

\theoremstyle{remark}

\newcommand{\be}{\begin{equation}}
\newcommand{\ee}{\end{equation}}
\numberwithin{equation}{section}

\begin{document}

\title{Explicit formulas for the Hattori-Stong theorem and applications}

\author{Ping Li}
\address{School of Mathematical Sciences, Fudan University, Shanghai 200433, China}

\email{pinglimath@fudan.edu.cn\\pinglimath@gmail.com}
\author{Wangyang Lin}

\address{School of Mathematical Sciences, Fudan University, Shanghai 200433, China}

\email{wylin23@m.fudan.edu.cn}
\thanks{This work was partially supported by the National
Natural Science Foundation of China (Grant No. 12371066).}
\subjclass[2010]{57R20, 05E05, 19L64, 32Q60.}


\keywords{stably almost-complex structure, Chern number, the Hattori-Stong theorem, the Hirzebruch signature formula, the Stirling number of the second kind, Bernoulli number, rational projective plane}

\begin{abstract}
We employ combinatorial techniques to present an explicit formula for the coefficients in front of Chern classes involving in the Hattori-Stong integrability conditions. We also give an evenness condition for the signature of stably almost-complex manifolds in terms of Chern numbers. As an application, it can be shown that the signature of a $2n$-dimensional stably almost-complex manifold whose possibly nonzero Chern numbers being $c_n$ and $c_ic_{n-i}$ is even, which particularly rules out the existence of such structure on rational projective planes. Some other related results and remarks are also discussed in this article.
\end{abstract}

\maketitle


\section{Introduction}\label{introduction}
Let $E$ and $F$ be two complex vector bundles over a topological space $X$, $\wedge^k(\cdot)$ the $k$-th exterior power operation and $$\wedge_t(E):=1+\sum_{k\geq1}\wedge^k(E)t^k \qquad(\text{$1:=$trivial line bundle}).$$
The operation $\wedge_t(\cdot)$ in $K$-theory satisfies
\be\label{wedge operation formula}\wedge_t(E-F)=
\wedge_t(E)[\wedge_t(F)]^{-1}\in K(X)[[t]].\ee
The $K$-theory operations $\gamma^k(\cdot)$ are defined by
\be\label{gamma operation formula}1+\sum_{k\geq1}\gamma^k(E-F)t^k:=
\wedge_{\frac{t}{1-t}}(E-F)\in K(X)[[t]].\ee
The importance of these operations
$\gamma^k(\cdot)$ can be appreciated from at least two aspects. Firstly, these $\gamma^k(\cdot)$ form power series generators over the integers for all $K$-theory operations (\cite[p.128]{At}). Secondly, if $\widetilde{E}:=E-\underline{\mathbb{C}}^{\text{dim}(E)}$, where
$\underline{\mathbb{C}}^{\text{dim}(E)}$ denotes the trivial bundle of rank $E$, then $\gamma^1(\widetilde{\xi_n}),\ldots,\gamma^n(\widetilde{\xi_n})$
generate over the integers the
$K$-ring of the classifying space $BU(n)$ in a suitable sense (\cite[p.253]{Kar}), where $\xi_n$ is the tautological $n$-plane bundle over $BU(n)$.

Let $(M,\tau)$ be a closed stably almost-complex manifold, where $\tau$ is a complex vector bundle over a closed smooth manifold $M$ such that the underlying real vector bundle $\tau_{\mathbb{R}}$ is isomorphic to the stable tangent bundle of $M$. The Chern classes of $M$ are defined to be those of $\tau$. Namely, $c_i(M):=c_i(\tau)\in H^{2i}(M;\mathbb{Z})$. When $\text{dim}(M)=2n$, with respect to the canonical orientation induced from $\tau$ the set of Chern numbers $\{c_{\lambda}[M]~|~\text{$\lambda:$ integer partitions of weight $n$}\}$ can be defined. It is well-known that two stably almost-complex manifolds are complex cobordant if and only if they have the same Chern numbers (\cite{Mi}).

The integrality of the following linear combinations of Chern numbers
\be\label{integrality}
\int_Mch\big(\gamma^{k_1}(\widetilde{\tau})\big)\cdots
ch\big(\gamma^{k_i}(\widetilde{\tau})\big)td(M)\in\mathbb{Z},
\qquad\forall~ k_1,\ldots,k_i\in\mathbb{Z}_{\geq0},\qquad
\gamma^0(\widetilde{\tau}):=1\ee
was first observed by Atiyah and Hirzebruch as corollaries of their differential Riemann-Roch theorem (\cite{AH1},\cite{AH2}), where $ch(\cdot)$ denotes the Chern character and $td(\cdot)$ the Todd polynomial. Now (\ref{integrality}) is a direct corollary of the Atiyah--Singer index theorem (\cite{AS}) as each $\gamma^k(\widetilde{E})$ is an \emph{integral} linear combination of those $\wedge^i(E)$ due to (\ref{wedge operation formula}) and (\ref{gamma operation formula}). Conversely, they conjectured in \cite{AH2} that (\ref{integrality}) give all the integral relations for Chern numbers, which was proved by Hattori and Stong independently (\cite{Ha},\cite{St1}) and is now called the Hattori--Stong theorem. More precisely, we have

\begin{theorem}[Atiyah-Hirzebruch, Hattori-Stong]\label{AHHS}
Given a positive integer $n$ and a set of integers $\{c_{\lambda}~|~\text{$\lambda:$ partitions of weight $n$}\}$, they can be realized as Chern numbers of some $2n$-dimensional stably almost-complex manifold if and only if they satisfy the integral conditions (\ref{integrality}).
\end{theorem}

There are also analogies to Theorem \ref{AHHS} for smooth and spin manifolds, in which the Hirzebruch's $L$-polynomial and the $\hat{A}$-polynomial are involved (\cite{St1},\cite{St2},\cite{St}). In order to effectively apply Theorem \ref{AHHS} and these related results in concrete problems, we need, when expanding these quantities in terms of Chern and Pontrjagin classes, \emph{both closed and explicit} formulas for coefficients in front of them. Recently some considerable progress towards it has been made. Berglund and Bergstr\"{o}m gave explicit formulas in terms of multiple zeta values for the coefficients in front of Pontrjagin classes for the $L$-polynomial and the $\hat{A}$-polynomial (\cite{BB}). Combining some ideas in \cite{BB} and other combinatorial tools, the first author explored in \cite{Li} the coefficients in front of Chern classes for the complex genera including the Todd polynomial. So in order to apply Theorem \ref{AHHS}, we need to know an explicit formula for the coefficients $b_{\lambda}^{(k)}$ when expanding $ch\big(\gamma^k(\widetilde{E})\big)$ in terms of the Chern monomials $c_{\lambda}(E)$:
\be\label{coefficient in terms of Chern character}ch\big(\gamma^k(\widetilde{E})\big)=
\sum_{\lambda}b_{\lambda}^{(k)}\cdot c_{\lambda}(E),\qquad
b_{\lambda}^{(k)}\in\mathbb{Q}.\ee

\emph{Our first main result} (Theorem \ref{first main result theorem}) in this article is to give an explicit closed formula for $b_{\lambda}^{(k)}$ in (\ref{coefficient in terms of Chern character}) by further developing the techniques in \cite{Li}. \emph{Our second main result} (Theorem \ref{second main result theorem}) is to give an evenness condition in terms of Chern numbers for the signature of stably almost-complex manifolds. To this end, still applying the tools in \cite{Li}, we shall give an explicit formula for the coefficients in front of Chern numbers when expressing the signature of a stably almost-complex manifold in terms of them. As an application, we combine the first two main results to show that the signature of a $2n$-dimensional stably almost-complex manifold whose possibly nonzero Chern numbers being $c_n$ and $c_ic_{n-i}$ for some $i$ is even (Theorem \ref{third main result theorem}), which particularly excludes the existence of such a structure on rational projective planes. Some other related results and remarks are also discussed at places along the main line of this article.

The rest of this article is organized as follows. The main results as well as some necessary background notation are stated in Section \ref{main results}. Some preliminaries on symmetric functions and Stirling numbers of the second kind are presented in Section \ref{preliminaries}. Then Sections \ref{Proof of first main result}, \ref{proof of second main result} and \ref{proof of third main result} are devoted to the proofs of Theorems \ref{first main result theorem}, \ref{second main result theorem} and \ref{third main result theorem} respectively. Some remarks on integral/rational projective planes are collected in Section \ref{remarks on RPP} for the reader's convenience.

\section{Main results}\label{main results}
\subsection{Background notation}\label{notation subsection}
Before stating the main results in this article, we introduce some more notation in combinatorics.

An \emph{integer partition} $\lambda=(\lambda_1,\lambda_2,\ldots,\lambda_l)$ is a finite sequence of positive integers in non-increasing order: $\lambda_1\geq\lambda_2\geq\cdots\geq\lambda_l>0$. Denote by $l(\lambda):=l$ and $|\lambda|:=\sum_{i=1}^l\lambda_i$ and they are called the \emph{length} and \emph{weight} of the partition $\lambda$ respectively. These $\lambda_i$ are called \emph{parts} of the integer partition $\lambda$. It is also convenient to use another notation which indicates the number of times each integer appears: $\lambda=(1^{m_1(\lambda)}2^{m_2(\lambda)}\ldots)$. This means that $i$ appears $m_i(\lambda)$ times among $\lambda_1,\ldots,\lambda_{l(\lambda)}$. For example, $\lambda=(6,5,5,4,2,2,1,1,1)=(1^32^23^04^15^26^1)$. With this notation the Chern monomials $c_{\lambda}(E)$ in (\ref{coefficient in terms of Chern character}) are defined by $c_{\lambda}(E):=\prod_{i=1}^{l(\lambda)}c_{\lambda_i}(E)$.

Given a (nonempty) finite set $S$, a \emph{partition of $S$} is a collection of disjoint nonempty subsets of $S$ whose union is exactly $S$ (\cite[p.33]{Sta97}). Namely, a partition $\pi$ of $S$ is of the form $\pi=\{\pi_1,\ldots,\pi_k\}$, where each $\pi_i$ is a nonempty subset of $S$, $\pi_i\cap\pi_j=\emptyset$ for any $i\neq j$, and $\cup_i\pi_i=S$. In this case, each $\pi_i$ is called a \emph{block} of $\pi$. Let $l(\pi)=k$ and call it the \emph{length} of $\pi$. As usually $|\pi_i|$ denotes the cardinality of $\pi_i$.

Put $[n]:=\{1,\ldots,n\}$ and denote by $\Pi_n$ the set consisting of all partitions of $[n]$. For instance, $\big\{\{1,3,6\},\{2\},\{4,5\}\big\}\in\Pi_6$, whose length is $3$.

Define $S(n,k)$ to be the number of partitions of $[n]$ into exactly $k$ blocks. $S(n,k)$ is usually called \emph{a Stirling number of the second kind} (\cite[p.33]{Sta97}). Then $S(n,k)>0$ if $1\leq k\leq n$. By convention we put
\be\label{notation for stirling number}\text{$S(n,k):=0$ if $k>n$},\qquad \text{$S(n,0):=0$ if $n>0$},\qquad S(0,0):=1.\ee
For example, $S(n,1)=1$, $S(n,2)=2^{n-1}-1$, $S(n,n-1)={n\choose 2}$ and $S(n,n)=1$.

Let $\lambda$ be an integer partition and consider $\Pi_{l(\lambda)}$, which consists of all partitions of the set $[l(\lambda)]=\{1,2,\ldots,l(\lambda)\}$. If a partition $\pi=\{\pi_1,\ldots,\pi_{l(\pi)}\}\in\Pi_{l(\lambda)}$, then denote by
\be\lambda_{\pi_i}:=\sum_{j\in\pi_i}\lambda_j.\qquad\big(1\leq i\leq l(\pi)\big)\nonumber\ee

Let $B_{i}\in\mathbb{Q}_{>0}$ ($i\geq 1$) be the  \emph{Bernoulli numbers} without sign (\cite[p.281]{MS}), which are defined by
\be\label{Bernoulli}\frac{x}{\sinh(x)}=:1+\sum_{i\geq1}
\frac{(-1)^{i}(2^{2i}-2)B_{i}}{(2i)!}x^{2i},\qquad (B_1=\frac{1}{6},~B_2=\frac{1}{30},~B_3=\frac{1}{42},\ldots)\ee

\subsection{Main results}\label{main results subsection}
With the notation above in mind, we now state in this subsection the main results of this article.

The first result is the following explicit closed formula for $b^{(k)}_{\lambda}$ in (\ref{coefficient in terms of Chern character}).
\begin{theorem}\label{first main result theorem}
Let $E$ be a complex vector bundle over some topological space. 
Fix an integer partition $\lambda=(\lambda_1,\ldots,\lambda_{l(\lambda)})=
(1^{m_1(\lambda)}2^{m_2(\lambda)}\cdots)$. Then the coefficients $b_{\lambda}^{(k)}$ ($k\geq1$) in (\ref{coefficient in terms of Chern character}) are given by
\be\label{first main result formula}
\begin{split}
&\sum_{k\geq1}b^{(k)}_{\lambda}\cdot t^k\\
=&
\frac{(-1)^{|\lambda|-l(\lambda)}}{\prod_{i\geq1}m_i(\lambda)!}
\sum_{\pi\in\Pi_{l(\lambda)}}
\bigg\{\Big[\prod_{i=1}^{l(\pi)}(|\pi_i|-1)!\Big]
\Big[\prod_{i=1}^{l(\pi)}
\sum_{j=0}^{\lambda_{\pi_i}-1}
\frac{S(\lambda_{\pi_i},\lambda_{\pi_i}-j)}
{{\lambda_{\pi_i}-1\choose j}\cdot j!}(-t)^{\lambda_{\pi_i}-1-j}\Big]t^{l(\pi)}\bigg\}.
\end{split}\ee
In other words, $b^{(k)}_{\lambda}$ is the coefficient in front of $t^k$ on the right-hand side of (\ref{first main result formula}).
\end{theorem}
\begin{remark}\label{remarkvanishing}
Note that
the highest degree of $t$ on the RHS of (\ref{first main result formula}) is
$$l(\pi)+\sum_{i=1}^{l(\pi)}(\lambda_{\pi_i}-1)=|\lambda|$$ and hence $b^{(k)}_{\lambda}=0$ whenever $k>|\lambda|$. More vanishing information on $b^{(k)}_{\lambda}$ can be seen in Lemma \ref{lemmavanishing}.
\end{remark}

The following example is a simple illustration of Theorem \ref{first main result theorem} and it can be easily checked.
\begin{example}\label{special values of b example}
When $l(\lambda)\leq2$, the values $b^{(k)}_{\lambda}$ given by (\ref{first main result formula}) are
\be\label{special values of b}
b^{(k)}_{(i)}=\frac{(-1)^{i-k}S(i,k)\cdot(k-1)!}{(i-1)!},\qquad
b^{(k)}_{(i,i)}=\frac{1}{2\cdot(2i-1)!},\qquad
b^{(k)}_{(i,j)}=\frac{(-1)^{i+j}}{(i+j-1)!}~(i>j).\ee
\end{example}

Let $\sigma(M)$ be the signature of a closed oriented smooth manifold $M$, which is zero unless $\text{dim}(M)=4k$ for some $k\in\mathbb{Z}_{>0}$. By Hirzebruch's signature theorem (\cite{Hi1}) $\sigma(M)$ is a rational linear combination of Pontrjagin numbers. In \cite[Thm 1]{BB} an explicit and closed formula in terms of (some variant of) multiple zeta values for the coefficients in front of these Pontrjagin numbers is presented. When $M$ admits a stably almost-complex structure, Pontrjagin numbers are determined by Chern numbers and so is $\sigma(M)$. Our following second result is an evenness condition for the $\sigma(M)$ in terms of Chern numbers via an explicit formula for the coefficients in front of them.
\begin{theorem}\label{second main result theorem}
Let \be\label{expression for h} h_{2i}:=\frac{(-1)^{i}\cdot2^{2i+1}\cdot(2^{2i-1}-1)\cdot B_i}{(2i)!}\qquad\text{and}\qquad h_{2i-1}:=0\qquad(i\geq1).\ee
Let $M$ be a $4k$-dimensional closed stably almost-complex manifold with Chern numbers $c_{\lambda}[M]$ and $$\sigma(M)=:\sum_{|\lambda|=2k}h_{\lambda}\cdot c_{\lambda}[M].$$
\begin{enumerate}
\item
The coefficients $h_{\lambda}$ are given by
\be\label{signature in terms of Chern numbers}
h_{\lambda}=\frac{(-1)^{l(\lambda)}}{\prod_{i\geq1}m_i(\lambda)!}
\sum_{\pi\in\Pi_{l(\lambda)}}\Big\{(-1)^{l(\pi)}
\prod_{i=1}^{l(\pi)}\big[(|\pi_i|-1)!\cdot h_{\lambda_{\pi_i}}\big]\Big\}.\ee


\item
If the Chern numbers $c_{\lambda}[M]$ satisfy
\be\label{second main result formula}
\{\lambda~|~c_{\lambda}[M]\neq0\}\subset\{\lambda~|~\text{$m_i(\lambda)=0$
or $1$ for all $i$}\},
\ee
then $\sigma(M)$ is even.
\end{enumerate}
\end{theorem}
\begin{remark}
\begin{enumerate}
\item
The condition (\ref{second main result formula}) means that if some Chern number $c_{\lambda}[M]\neq0$, then all the parts $\lambda_i$ of this integer partition $\lambda$ must be mutually distinct.

\item
By (\ref{signature in terms of Chern numbers}), the coefficient in front of the top Chern number $c_{2k}[M]$ is $h_{(2k)}=h_{2k}$, which was known to Hirzebruch (\cite[Formula $(8)$]{Hi53}). Before the work in \cite{BB}, Fowler and Su presented a \emph{recursive} formula in \cite[Appendix]{FS} for the coefficients of $\sigma(M)$ in front of Pontrjagin numbers.
\end{enumerate}
\end{remark}

An $n$-dimensional smooth closed connected orientable manifold whose Betti numbers satisfy
$b_0=b_{\frac{n}{2}}=b_{n}=1$ and $b_i=0$ for other $i$ is called a \emph{rational projective plane} (RPP for short).
Classical such manifolds include complex, quaternionic and octonionic projective planes $\mathbb{C}P^2$, $\mathbb{H}P^2$ and $\mathbb{O}P^2$, whose dimensions are $4$, $8$ and $16$ respectively. Various constraints on the possibly existing dimensions on RPP were studied by Adem and Hirzebruch (\cite{Ade},\cite{Hi53},\cite{Hi54-1},\cite{Hi54-2}) and more details can be found in Section \ref{remarks on RPP}. It turns out that such existence problem is closely related to that of manifolds with prescribed Betti numbers (see Proposition \ref{Betti restrictions}). Recently the existence problem of RPP was systematically treated by Zhixu Su and her coauthors (\cite{Su1},\cite{Su2},\cite{KS},\cite{FS}). Moreover Su (\cite{Su3}) and Hu (\cite{Hu}) independently showed that any RPP whose dimension is larger than $4$ cannot admit any almost-complex structure.

Note that the signature of any RPP is $\pm1$ and hence odd. As a simple application of Theorems \ref{first main result theorem} and \ref{second main result theorem}, we shall show the following result, which particularly excludes the existence of \emph{stably almost-complex structure} on any RPP whose dimension is larger than $4$.
\begin{theorem}\label{third main result theorem}
Let $M$ be a $4k$\rm($k>1$\rm)-dimensional stably almost-complex manifold whose possibly nonzero Chern numbers are $c_{2k}[M]$ and $c_{i}c_{2k-i}[M]$ for some $0<i<2k$. Then $\sigma(M)$ is even.
\end{theorem}
\begin{remark}
We would like to stress at this moment that the notion ``stably almost-complex" is strictly weaker than that of ``almost-complex". For instance, it is well-known that the sphere $S^n$ admits an almost-complex structure if and only if $n=2$ or $6$ (\cite{BS}). Nevertheless, \emph{any} $S^n$ does admit a stably almost-complex structure as its tangent bundle is stably trivial.
\end{remark}

\section{Preliminaries}\label{preliminaries}
\subsection{Symmetric functions}\label{symmetric function}
In this subsection we briefly recall three bases of the vector space consisting of symmetric functions and collect several facts related to them, which we shall apply to prove Theorem \ref{first main result theorem}. Two standard references on symmetric function theory are \cite[\S 1]{Ma95} and \cite[\S 7]{Sta99}.

Let $\mathbb{Q}[[\mathbf{x}]]=\mathbb{Q}[[x_1,x_2,\ldots]]$ be the ring of formal power series over $\mathbb{Q}$ in a countably infinite set of (commuting) variables $\mathbf{x}=(x_1,x_2,\ldots)$. An $f(\mathbf{x})=f(x_1,x_2,\ldots)\in \mathbb{Q}[[\mathbf{x}]]$ is called a \emph{symmetric function} if it satisfies
$$f(x_{\sigma(1)},x_{\sigma(2)},\ldots)=f(x_1,x_2,\ldots),
\qquad\forall~\sigma\in S_k,~\forall~k\in\mathbb{Z}_{>0}.$$
Here, if $\sigma\in S_k$, $\sigma(i)=i$ for $i>k$ is understood.

Let $\Lambda^k(\mathbf{x})$ be the vector space of symmetric functions of homogeneous degree $k$ \big($\text{deg}(x_i):=1$\big). Then the \emph{ring} of symmetric functions $\Lambda(\mathbf{x}):=\bigoplus_{k=0}^{\infty}\Lambda^k(\mathbf{x})$ ($\Lambda^0(\mathbf{x}):=\mathbb{Q}$)
consists of all symmetric functions with \emph{bounded} degree.

\begin{definition}
Let $\lambda=(\lambda_1,\ldots,\lambda_{l(\lambda)})$ be an integer partition with $|\lambda|=n$.
The $k$-th \emph{elementary} symmetric function $e_k(\mathbf{x})$ and \emph{power sum} symmetric function $p_k(\mathbf{x})$ are defined respectively by
$$e_k(\mathbf{x}):=\sum_{1\leq i_1<i_2<\cdots<i_k}x_{i_1}x_{i_2}\cdots x_{i_k},\qquad p_k(\mathbf{x}):=\sum_{i=1}^{\infty}x_i^k,$$
and
\be\label{e and p-def}e_{\lambda}(\mathbf{x}):=
\prod_{i=1}^{l(\lambda)}e_{\lambda_i}(\mathbf{x})\in\Lambda^n(\mathbf{x}),
\qquad p_{\lambda}(\mathbf{x}):=
\prod_{i=1}^{l(\lambda)}p_{\lambda_i}(\mathbf{x})\in\Lambda^n(\mathbf{x}).\ee

The \emph{monomial} symmetric function $m_{\lambda}(\mathbf{x})\in\Lambda^n(\mathbf{x})$ is defined by
$$m_{\lambda}(\mathbf{x}):=
\sum_{(\alpha_1,\alpha_2,\ldots)}x_1^{\alpha_1}x_2^{\alpha_2}\cdots,$$
where the sum is over all \emph{distinct} permutations $(\alpha_1,\alpha_2,\ldots)$ of the entries of the vector $(\lambda_1,\ldots,\lambda_{l(\lambda)},0,\ldots)$. In other words, $m_{\lambda}(\mathbf{x})$ is the \emph{smallest} symmetric function containing the monomial $x_1^{\lambda_1}x_2^{\lambda_2}\cdots x_{l(\lambda)}^{\lambda_{l(\lambda)}}.$
\end{definition}

It is well-known that (\cite[\S 1.2]{Ma95}) the three sets
$\big\{e_{\lambda}(\mathbf{x})~\big|~|\lambda|=n\big\}$, $\big\{p_{\lambda}(\mathbf{x})~\big|~|\lambda|=n\big\}$ and $\big\{m_{\lambda}(\mathbf{x})~\big|~|\lambda|=n\big\}$
are all additive bases of the vector space $\Lambda^n(\mathbf{x})$, and so it is natural to ask what the transition matrices are among these bases. In \cite{Do72} Doubilet applied the combinatorial M\"{o}bius inversion tools to give a unified and compact treatment on the entries of the transition matrices among these three bases as well as other bases (complete symmetric function, forgotten symmetric function and so on). The following transition relation will be used in our later proof (\cite[Thm 2]{Do72}).
\begin{theorem}[Doubilet]\label{transtion matrice from m to p theorem}
Let $\lambda=(\lambda_1,\ldots,\lambda_{l(\lambda)})=
(1^{m_1(\lambda)}2^{m_2(\lambda)}\cdots)$ be an integer partition. Then we have
\be\label{transition matrice from m to p formula}m_{\lambda}(\mathbf{x})=
\frac{(-1)^{l(\lambda)}}{\prod_im_i(\lambda)!}\sum_{\pi\in\Pi_{l(\lambda)}}
\Bigg\{(-1)^{l(\pi)}
\prod_{i=1}^{l(\pi)}\Big[\big(|\pi_i|-1\big)!\cdot
p_{\lambda_{\pi_i}}(\mathbf{x})\Big]\Bigg\},\ee
where the related notation involved in (\ref{transition matrice from m to p formula}) has been introduced in Section \ref{notation subsection}.
\end{theorem}
\begin{remark}
This formula has intimate relation with a well-known multiple zeta value formula due to Hoffman (see \cite[\S 3]{BB}) and this phenomenon has been extended in \cite[\S 3]{Li}.
\end{remark}

The following fact is well-known and usually attributed to Cauchy (\cite[\S 1.4]{Hi1} or \cite[Lemma 4.1]{Li})
\begin{lemma}\label{technical lemma}
Let $Q(x)=1+\sum_{i\geq1}a_ix^i$ be a formal power series.
If $a_i$ are viewed as $a_i=e_i(\mathbf{y})=e_i(y_1,y_2,\ldots),$
the $i$-th elementary symmetric functions of the variables $y_1,y_2,\ldots,$ then the power sum symmetric functions $p_i(\mathbf{y})$ are determined by $a_i=e_i(\mathbf{y})$ via
\be\label{cauchy}
\sum_{i\geq1}(-1)^{i-1}\cdot p_i(\mathbf{y})\cdot x^{i-1}=\frac{Q'(x)}{Q(x)}.
\ee
\end{lemma}

The following identity will also be used in the sequel (\cite[p.292]{Sta99}).
\begin{lemma}
Let $\mathbf{x}=(x_1,x_2,\ldots)$ and $\mathbf{y}=(y_1,y_2,\ldots)$ be two countably infinite sets of (commuting) variables $x_i$ and $y_j$. Then we have
\be\label{formula m-e relation}
\prod_{i,j\geq1}(1+x_iy_j)=1+\sum_{|\lambda|\geq 1}m_{\lambda}(\mathbf{y})
e_{\lambda}(\mathbf{x}),
\ee
where the sum is over all positive integer partitions.
\end{lemma}

\subsection{Stirling numbers of the second kind}
Recall the definition of the Stirling numbers of the second kind $S(n,k)$ and their convention in Section \ref{notation subsection}. For $S(n,k)$ we have the following two formulas.
\begin{lemma}\label{identity for Stirling numbers lemma}
These $S(n,k)$ satisfy
\be\label{identity for Stirling numbers formula1}
\frac{(e^x-1)^k}{k!}=\sum_{n\geq0}S(n,k)\cdot\frac{x^n}{n!},\qquad\forall~k\geq0,
\ee
and
\be
\label{identity for Stirling numbers formula2}
S(n+1,k+1)=\sum_{k\leq l\leq n}{n\choose l}S(l,k)\xlongequal{(\ref{notation for stirling number})}\sum_{0\leq l\leq n}{n\choose l}S(l,k).
\ee
\end{lemma}
\begin{proof}
The identity (\ref{identity for Stirling numbers formula1}) is well-known (\cite[p.34]{Sta97}). For (\ref{identity for Stirling numbers formula2}), we differentiate both sides of (\ref{identity for Stirling numbers formula1}) with respect to $x$:
\be\label{identity}
\sum_{n\geq0}S(n,k)\cdot\frac{x^{n-1}}{(n-1)!}
=e^x\cdot\frac{(e^x-1)^{k-1}}{(k-1)!}
\xlongequal{(\ref{identity for Stirling numbers formula1})}
\sum_{l,m\geq0}S(l,k-1)\cdot\frac{x^{l+m}}{l!\cdot m!}.
\ee
Comparing the coefficient in front of $x^{n-1}$ on both sides of (\ref{identity}) yields
$$S(n,k)=\sum_{k-1\leq l\leq n-1}{n-1\choose l}S(l,k-1),$$
which is $(\ref{identity for Stirling numbers formula2})$.
Alternatively, we can give a direct combinatorial proof for (\ref{identity for Stirling numbers formula2}) as follows. There are $S(n+1,k+1)$ ways to partition $[n+1]$ into $k+1$ blocks. We consider the block containing the element $n+1$. Assume that this block has $l+1$ elements ($l\geq0$). We have ${n\choose l}$ ways to choose this block and in each way we can partition the resulting $n-l$ elements into $k$ blocks in $S(n-l,k)$ ways. So
$$S(n+1,k+1)=\sum_{0\leq l\leq n}{n\choose l}S(n-l,k)=
\sum_{k\leq l\leq n}{n\choose l}S(l,k).$$
\end{proof}

\section{Proof of Theorem \ref{first main result theorem}}\label{Proof of first main result}
Let $\mathbf{x}=(x_1,x_2,\ldots)$ be the \emph{formal Chern roots} of a complex vector bundle $E$. Usually finite variables whose number is no less than the rank of $E$ are enough, but here in order to being consistent with symmetric function theory, we take infinite variables instead. This means that  $e_k(\mathbf{x})=c_k(E)$. In this notation it turns out that (\cite[p.268]{St1})
\be\label{symmetric function of gamma}
ch\big(\gamma^k(\widetilde{E})\big)=e_k(e^{\mathbf{x}}-1):=
e_k(e^{x_1}-1,e^{x_2}-1,\ldots)
\ee
and hence determining the coefficients $b^{(k)}_{\lambda}$ in (\ref{coefficient in terms of Chern character}) can be translated into a pure combinatorial problem, i.e., the coefficients $b^{(k)}_{\lambda}$ determined by
\be\label{coefficient in terms of symmetric function}e_k(e^{\mathbf{x}}-1)=:
\sum_{\lambda}b_{\lambda}^{(k)}\cdot e_{\lambda}(\mathbf{x}).\ee

Note that the exponential function $e^x$ is a monic formal power series. Thus we may broaden our scope by considering the following question, which may be of independent interest in combinatorics and related topics.
\begin{question}\label{question}
Let $Q(x)=1+\sum_{i\geq1}a_i\cdot x^i$ be a monic formal power series. Consider the symmetric function $e_k\big(Q(\mathbf{x})-1\big):=e_k\big(Q(x_1)-1,Q(x_2)-1,\ldots\big)$
and
\be\label{coefficient for general}
e_k\big(Q(\mathbf{x})-1\big)=:\sum_{\lambda}b^{(k)}_{\lambda}(Q)\cdot e_{\lambda}(\mathbf{x}).
\ee
What is the closed formula for $b^{(k)}_{\lambda}(Q)$?
\end{question}

Our solution to Question \ref{question} is

\begin{theorem}\label{solution to Question theorem}
Let the notation be as above and $Q_t(x):=1+\sum_{i\geq1}(ta_i)x^i$. For each integer partition $\lambda=(\lambda_1,\ldots,\lambda_{l(\lambda)})=
(1^{m_1(\lambda)}2^{m_2(\lambda)}\cdots)$ we have
\be\label{solution to Question formula}
\sum_{k\geq1}b^{(k)}_{\lambda}(Q)\cdot t^k=
\frac{(-1)^{|\lambda|-l(\lambda)}}{\prod_{i\geq1}m_i(\lambda)!}
\sum_{\pi\in\Pi_{l(\lambda)}}
\bigg\{\big[\prod_{i=1}^{l(\pi)}(|\pi_i|-1)!\big]
\Big[\prod_{i=1}^{l(\pi)}\big(\frac{1}{2\pi\sqrt{-1}}
\oint\frac{Q'_t(x)}{x^{\lambda_{\pi_i}}Q_t(x)}dx\big)
\Big]\bigg\}.
\ee
\end{theorem}

\begin{proof}
View $ta_i=:e_i(\mathbf{y})$, the $i$-th elementary symmetric function of the variables $\mathbf{y}=(y_1,y_2,\ldots)$. Then
\be\label{id1}
Q_t(x)=\prod_{i\geq1}(1+y_ix)
\ee
and
\be\label{id2}\begin{split}
1+\sum_{|\lambda|\geq1}\big(\sum_{k\geq1}b^{(k)}_{\lambda}(Q)\cdot t^k\big)e_{\lambda}(\mathbf{x})&=
1+\sum_{k\geq1}e_k\big(Q(\mathbf{x})-1\big)t^k
\qquad\big(\text{by (\ref{coefficient for general})}\big)\\
&=\prod_{i\geq1}\bigg[1+\big(Q(x_i)-1\big)t\bigg]\\
&=\prod_{i\geq1}Q_t(x_i)\\
&=\prod_{i,j\geq1}(1+x_iy_j)\qquad\big(\text{by (\ref{id1})}\big)\\
&=1+\sum_{|\lambda|\geq1}m_{\lambda}(\mathbf{y})e_{\lambda}(\mathbf{x}).
\qquad\big(\text{by (\ref{formula m-e relation})}\big)
\end{split}\ee
Comparing both sides of (\ref{id2}) yields
\be\label{id3}\begin{split}
\sum_{k\geq1}b^{(k)}_{\lambda}(Q)\cdot t^k&=m_{\lambda}(\mathbf{y})\\
&=\frac{(-1)^{l(\lambda)}}{\prod_im_i(\lambda)!}\sum_{\pi\in\Pi_{l(\lambda)}}
\Bigg\{(-1)^{l(\pi)}
\prod_{i=1}^{l(\pi)}\Big[\big(|\pi_i|-1\big)!\cdot
p_{\lambda_{\pi_i}}(\mathbf{y})\Big]\Bigg\}.\quad\big(\text{by (\ref{transition matrice from m to p formula})}\big)
\end{split}\ee
By (\ref{cauchy}) we have
\be
\begin{split}
(-1)^{\lambda_{\pi_i}-1}p_{\lambda_{\pi_i}}(\mathbf{y})&=
\text{the coefficient of $x^{\lambda_{\pi_i}-1}$ in $\frac{Q'_t(x)}{Q_t(x)}$}\\
&=\frac{1}{2\pi\sqrt{-1}}
\oint\frac{Q'_t(x)}{x^{\lambda_{\pi_i}}Q_t(x)}dx\qquad
(\text{the residue formula})
\end{split}\nonumber
\ee
and therefore
\be\label{id4}
\prod_{i=1}^{l(\pi)}
p_{\lambda_{\pi_i}}(\mathbf{y})=
(-1)^{|\lambda|-l(\pi)}\prod_{i=1}^{l(\pi)}\big(\frac{1}{2\pi\sqrt{-1}}
\oint\frac{Q'_t(x)}{x^{\lambda_{\pi_i}}Q_t(x)}dx\big).
\ee
Putting (\ref{id4}) into (\ref{id3}) leads to the desired (\ref{solution to Question formula}).
\end{proof}

In view of (\ref{first main result formula}) and (\ref{solution to Question formula}), Theorem \ref{first main result theorem} follows from the following
\begin{lemma}
We have
\be\label{id5}
\frac{1}{2\pi\sqrt{-1}}
\oint\frac{Q'_t(x)}{x^kQ_t(x)}dx
\xlongequal[]{Q(x)=e^x}t\sum_{i=0}^{k-1}\frac{S(k,k-i)}{{k-1\choose i}\cdot i!}(-t)^{k-1-i}.\nonumber
\ee
\begin{proof}
In this case $Q_t(x)=1+t(e^x-1)$. For simplicity we use $\langle x^i\rangle f(x)$ to denote the coefficient in front of $x^i$ in $f(x)$. Then
\be\begin{split}
\frac{1}{2\pi\sqrt{-1}}
\oint\frac{Q'_t(x)}{x^kQ_t(x)}dx&=\frac{1}{2\pi\sqrt{-1}}
\oint\frac{te^x}{x^k\big[1+t(e^x-1)\big]}dx\\
&=\langle x^{k-1}\rangle te^x\big[1+t(e^x-1)\big]^{-1}\\
&=\langle x^{k-1}\rangle te^x\sum_{i=0}^{k-1}(-t)^{k-1-i}(e^x-1)^{k-1-i}\\
&=t\sum_{i=0}^{k-1}\big[(-t)^{k-1-i}\cdot\langle x^{k-1}\rangle e^x(e^x-1)^{k-1-i}\big]\\
&=t\sum_{i=0}^{k-1}\big[(-t)^{k-1-i}(k-1-i)!\sum_{0\leq l\leq k-1}\frac{S(l,k-1-i)}{l!\cdot(k-1-l)!}
\big]\qquad\big(\text{by (\ref{identity for Stirling numbers formula1})}\big)\\
&=t\sum_{i=0}^{k-1}\big[\frac{(-t)^{k-1-i}(k-1-i)!}
{(k-1)!}\sum_{0\leq l\leq k-1}{k-1\choose l}S(l,k-1-i)
\big]\\
&=t\sum_{i=0}^{k-1}\big[\frac{(-t)^{k-1-i}(k-1-i)!}
{(k-1)!}S(k,k-i)\big]\qquad\big(\text{by (\ref{identity for Stirling numbers formula2})}\big)\\
&=t\sum_{i=0}^{k-1}\frac{S(k,k-i)}{{k-1\choose i}\cdot i!}(-t)^{k-1-i}.
\end{split}\nonumber\ee
\end{proof}
\end{lemma}

Before ending this section, we collect several general vanishing information on $b^{(k)}_{\lambda}$ in the following lemma.
\begin{lemma}\label{lemmavanishing}
We have $b^{(k)}_{\lambda}=0$ unless $|\lambda|\geq k$, $b^{(k)}_{(k)}=1$, and $b^{(k)}_{\lambda}=0$
whenever $|\lambda|=k$ and $l(\lambda)\geq2$.
\end{lemma}
\begin{proof}
The first one has been explained in Remark \ref{remarkvanishing}. To see the latter two cases it suffices to note that
$$e_k(e^{x_1}-1,e^{x_2}-1,\ldots)
=e_k(x_1,x_2,\ldots)+\text{higher degree terms}.$$
\end{proof}

\section{Proof of Theorem \ref{second main result theorem}}\label{proof of second main result}
Let $(M,\tau)$ be a closed $2n$-dimensional stably almost-complex manifold. Recall that the value of a \emph{complex genus} $\varphi$ which corresponds to a monic formal power series $Q(x)=1+\sum_{i\geq1}a_ix^i$ $(a_i\in\mathbb{Q})$ on $M$, $\varphi(M)$, is a  rational linear combination of Chern numbers $c_{\lambda}[M]$ and defined as follows (\cite[\S 1.8]{HBJ}).
Let 
\be\label{formal decomposition1}\prod_{i\geq1}Q(x_i)=:1+\sum_{j\geq1}Q_j\big(e_1(\mathbf{x}),\cdots,e_j(\mathbf{x})\big)
,\nonumber\ee
where $Q_j\big((e_1(\mathbf{x}),\cdots,e_j(\mathbf{x})\big)$, viewed as a polynomial of $e_1(\mathbf{x}),\cdots,e_j(\mathbf{x})$, denotes the homogeneous part of degree $j$ in $\prod_{i\geq1}Q(x_i)$ \big($\text{deg}(x_i)=1$\big). Then
\be\label{formal decomposition2}\varphi(M)=\int_MQ_n\big(c_1(M),\cdots,c_n(M)\big)
=:h_n(\varphi)c_n[M]+\sum_{|\lambda|= n,~l(\lambda)\geq2}h_{\lambda}(\varphi)c_{\lambda}[M].\ee

The coefficients $h_{\lambda}(\varphi)$ in (\ref{formal decomposition2}) can be expressed in terms of $h_i(\varphi)$ as follows (cf. \cite[Lemma 4.3]{Li}).
\begin{lemma}
Let $\varphi$ be a complex genus which corresponds to the formal power series $Q(x)$.
Then the coefficients $h_{\lambda}(\varphi)$ in (\ref{formal decomposition2}) are given by
\be\label{formula for coefficient}
h_{\lambda}(\varphi)=\frac{(-1)^{l(\lambda)}}{\prod_im_i(\lambda)!}\sum_{\pi\in\Pi_{l(\lambda)}}
\Bigg\{(-1)^{l(\pi)}\prod_{i=1}^{l(\pi)}\Big[\big(|\pi_i|-1\big)!\cdot
h_{\lambda_{\pi_i}}(\varphi)\Big]\Bigg\},\ee
where $h_i(\varphi)$ are determined by
\be\label{id6}
\sum_{i\geq1}(-1)^{i-1}\cdot h_i(\varphi)\cdot x^{i-1}=\frac{Q'(x)}{Q(x)}.
\ee
\end{lemma}

Now the complex genus which corresponds to the formal power series $Q(x)=x/\tanh(x)$ is nothing but the signature $\sigma(\cdot)$ due to Hirzebruch's signature theorem (\cite{Hi1}). Therefore
\be\label{id7}\begin{split}
\frac{Q'(x)}{Q(x)}\xlongequal{Q(x)=\frac{x}{\tanh(x)}}&
\frac{1}{x}\big[1-\frac{2x}{\sinh(2x)}\big]\\
=&\sum_{i\geq1}\frac{(-1)^{i-1}\cdot2^{2i+1}\cdot(2^{2i-1}-1)\cdot B_i}{(2i)!}x^{2i-1}.
\qquad\big(\text{by (\ref{Bernoulli})}\big)
\end{split}\ee
Comparing (\ref{id6}) and (\ref{id7}) leads to
$$h_{2i-1}(\sigma)=0,\qquad h_{2i}(\sigma)=\frac{(-1)^{i}\cdot2^{2i+1}\cdot(2^{2i-1}-1)\cdot B_i}{(2i)!},$$
which, together with (\ref{formula for coefficient}), yields (\ref{signature in terms of Chern numbers}) and hence finishes the proof of the first part in Theorem \ref{second main result theorem}.

To prove the second part in Theorem \ref{second main result theorem}. We need some more notation and facts.

The \emph{$2$-adic valuation} of an integer $k$ is defined to be
\begin{eqnarray}\label{2-adic}
\nu_2(k):=
\left\{ \begin{array}{ll}
\displaystyle
\max\{t\in\mathbb{Z}_{\geq0}:~\text{$2^t$ divides $k$}\}&\text{if $k\neq0$}\\

\displaystyle
\infty&\text{if $k=0$}.
\end{array} \right.\nonumber
\end{eqnarray}
Clearly $\nu_2(k)=0$ if and only if $k$ is odd, $\nu_2(\cdot)$ satisfies $\nu_2(k_1k_2)=\nu_2(k_1)+\nu_2(k_2)$ and
\be\label{2-adic difference}\nu_2(k_1/k_2):=\nu_2(k_1)-\nu_2(k_2)\qquad (k_2\neq0)\nonumber\ee
is well-defined.

\begin{lemma}\label{2-adic of h}


Let $\lambda$ be an integer partition such that $m_i(\lambda)=0$ or $1$ for \emph{all} $i$, i.e., all the parts $\lambda_i$ of this $\lambda$ are mutually distinct, then
the $2$-adic valuation of the rational number $h_{\lambda}$ given in (\ref{signature in terms of Chern numbers}) satisfies $\nu_2(h_{\lambda})\geq1$.
\end{lemma}
\begin{proof}
Write the Bernoulli numbers $B_i$ as $B_i=:N_i/D_i$ ($N_i,D_i\in\mathbb{Z}_{>0}$) in the irreducible form. Since $D_i$ is equal to the product of all primes $p$ such that $p-1$ divides $2i$ (\cite[p. 284]{MS}), we have that $\nu_2(D_i)=1$ and hence $\nu_2(N_i)=0$.

Let $wt(i)$ denote the number of nonzero terms in the binary expansion of $i$. Then we have $\nu_2(i!)+wt(i)=i$, which is indeed a special case of the Legendre's formula (\cite[Thm 2.6.4]{Mo}, also cf. \cite[Lemma 2.2]{FS}). This implies that
\be\nu_2\big((2i)!\big)=2i-wt(2i)=2i-wt(i),\nonumber\ee
and thus \be\label{2-adic h2i}\nu_2(h_{2i})\xlongequal[]{(\ref{expression for h})}
(2i+1)+\nu_2(N_i)-\nu_2\big((2i)!\big)-\nu_2(D_i)=wt(i)\geq1.\ee

The condition of $m_i(\lambda)=0$ or $1$ for all $i$ is equivalent to $$\nu_2\big(\prod_{i\geq1}m_i(\lambda)!\big)=0,$$
which, together with (\ref{2-adic h2i}), yields that for such $\lambda$ the coefficient $h_{\lambda}$ given by (\ref{signature in terms of Chern numbers}) satisfies $\nu_2(h_{\lambda})\geq1$.
\end{proof}

Now we can finish the proof of the second part in Theorem \ref{second main result theorem}.
\begin{proof}
If all nonzero Chern numbers $c_{\lambda}[M]$ are such that the integer partitions $\lambda$ satisfy $m_i(\lambda)=0$ or $1$ for all $i$, as required by the assumption (\ref{second main result formula}), every nonzero summand $h_{\lambda}c_{\lambda}[M]$ in the signature formula $$\sigma(M)=\sum_{|\lambda|=2k}h_{\lambda}c_{\lambda}[M]$$
satisfies $\nu_2(h_{\lambda}c_{\lambda}[M])\geq1$ due to Lemma \ref{2-adic of h}. This implies that $\nu_2\big(\sigma(M)\big)\geq1$, i.e., $\sigma(M)$ is even.
\end{proof}

\section{Proof of Theorem \ref{third main result theorem}}\label{proof of third main result}
The following congruence property will be used in proving Theorem \ref{third main result theorem} for the case of $i=k=2$, which may be of interest on its own.
\begin{lemma}\label{signaturetopchernnumber}
Let $M$ be a $4k$-dimensional stably almost-complex manifold.
The signature $\sigma(M)$ and the top Chern number $c_{2k}[M]$ satisfy
\be\label{congruence}\sigma(M)\equiv(-1)^kc_{2k}[M]~(\text{mod $4$}).\ee
\end{lemma}
\begin{proof}
First we have, for any $4k$-dimensional almost-complex manifold $N$,
\be\label{congruence2}\sigma(N)\equiv(-1)^k\chi(N)~(\text{mod $4$}),\ee
where $\chi(N)$ is the Euler characteristic of $N$.
This fact can be proved by using various symmetric properties of the Hirzebruch $\chi_y$-genus and appears in \cite[p.777]{Hi2}. A detailed proof can be found in \cite[(3.5),(3.7)]{RY}.
Secondly, any $4k$-dimensional stably almost-complex $M$ is complex cobordant to a $4k$-dimensional almost-complex manifold $N$ (\cite[Cor.5]{Kah}). Since Chern numbers (and hence signature) are complex cobordant invariants, we have $\sigma(M)=\sigma(N)$ and $c_{2k}[M]=c_{2k}[N]=\chi(N)$. Substituting them into (\ref{congruence2}) yields (\ref{congruence}).
\end{proof}

We now proceed to show Theorem \ref{third main result theorem} and therefore assume that the possibly nonzero Chern numbers of the $4k$-dimensional stably almost-complex manifold $(M,\tau)$ are $c_{2k}[M]$ and $c_ic_{2k-i}[M]$ for some $1\leq i\leq k$.

If $1\leq i\leq k-1$, then $\sigma(M)$ is even due to Theorem \ref{second main result theorem} as the assumption condition (\ref{second main result formula}) is satisfied.

If $i=k$, the possibly nonzero Chern numbers are $c_{2k}[M]$ and $c_{k}^2[M]$ and so the relevant Chern classes we need to deal with are $c_{2k}(M)$ and $c_{k}(M)$. In this situation 
\be\label{id9}
ch\big(\gamma^1(\widetilde{E})\big)\xlongequal[]{(\ref{coefficient in terms of Chern character})}b^{(1)}_{(k)}c_k(M)+
b^{(1)}_{(k,k)}c^2_k(M)+b^{(1)}_{(2k)}c_{2k}(M)+\text{other terms}\ee
and thus by Theorem \ref{AHHS} we have
\be\mathbb{Z}\ni\int_M\Big[ch\big(\gamma^1(\widetilde{E})\big)
ch\big(\gamma^1(\widetilde{E})\big)td(M)\Big]\xlongequal{(\ref{id9})}
\big(b^{(1)}_{(k)}\big)^2c_k^2[M]\xlongequal{(\ref{special values of b})}\frac{c_k^2[M]}{\big[(k-1)!\big]^2}.
\nonumber\ee
This implies that when $i=k\geq3$, the Chern number $c_k^2[M]$ is divisible by $4$ and therefore
\be\sigma(M^{4k})=h_{(2k)}c_{2k}[M]+h_{(k,k)}c^2_k[M]
\xlongequal{(\ref{signature in terms of Chern numbers})}
h_{2k}\cdot c_{2k}[M]+\frac{h^2_k-h_{2k}}{2}c^2_k[M]\nonumber\ee
is also even due to the facts that $h_{2i-1}=0$ and $\nu_2(h_{2i})\geq1$ from (\ref{2-adic h2i}).

It suffices to show the case $i=k=2$, i.e., the case of $8$-dimensional stably almost-complex manifold $M^8$ with possibly nonzero Chern numbers $c_4:=c_4[M]$ and $c^2_2:=c^2_2[M]$. In this dimension
\be\label{id8}
\sigma(M^8)=\frac{1}{45}(14c_4+3c_2^2),\qquad
\int_Mtd(M)=\frac{1}{720}(-c_4+3c_2^2)
\ee
and so $45\sigma(M^8)=14c_4+3c_2^2$, which, together with (\ref{congruence}), yields
\be\label{id10}
c_4\equiv c^2_2~(\text{mod $4$})
\ee
If $\sigma(M^8)$ is odd, then $c_4=2\alpha+1$ ($\alpha\in\mathbb{Z}$), still due to (\ref{congruence}). Then by the integrality of the Todd number $\int_Mtd(M)$ in (\ref{id8}) we have
\be\begin{split}
\mathbb{Z}\ni\int_Mtd(M)=\frac{1}{720}(-c_4+3c_2^2)&=
\frac{-(2\alpha+1)+3(2\alpha+1+4\beta)}{720}\qquad
\big(\text{by (\ref{id10}), $\beta\in\mathbb{Z}$}\big)\\
&=\frac{4\alpha+12\beta+2}{720},\end{split}\ee
which is a contradiction. Thus $\sigma(M^8)$ is even and completes the proof of Theorem \ref{third main result theorem}.

\section{Remarks on integral/rational projective planes}\label{remarks on RPP}
An $n$-dimensional smooth closed connected orientable manifold $M$ whose integral cohomology groups (resp. Betti numbers) satisfy
$H^i(M;\mathbb{Z})=\mathbb{Z}$ for $i=0,n/2, n$ and $H^i(M;\mathbb{Z})=0$ for other $i$ (resp. $b_0=b_{\frac{n}{2}}=b_{n}=1$ and $b_i=0$ for other $i$) is called an \emph{integral projective plane} (resp. a \emph{rational projective plane}) (IPP or RPP for short). As mentioned in Section \ref{main results subsection}, such (classical) examples exist when $n=4$, $8$ or $16$. The study of the existence/nonexistence of higher-dimensional IPP or RPP was initiated by Adem and Hirzebruch (\cite{Ade},\cite{Hi53},\cite{Hi54-1},\cite{Hi54-2}). Now it turns out that an IPP can \emph{only} exist when $n=4$, $8$ or $16$ (\cite[p.766]{Hi1.5}), which is a consequence of Adams' solution to the non-existence of Hopf invariant one (\cite{Ada}). Hirzebruch showed that (\cite[\S 3, Thm 1]{Hi54-1}) an $n$-($n>8$)dimensional RPP exists only if $n$ is of the form
$n=8(2^a+2^b)$ with $a,b\in\mathbb{Z}_{\geq0}$, i.e., $n$ is divisible by $8$ and $wt(n)\leq2$ (recall the notation in the proof of Lemma \ref{2-adic of h}). Moreover, a \emph{spin} RPP can only exist when $n=8$ or $16$ (\cite[p.786]{Hi1.5}, \cite[Thm C]{KS}).

Recently, the existence of RPP was investigated systematically by Su and her coauthors in a series of papers (\cite{Su1}, \cite{Su2}, \cite{FS}, \cite{KS}), whose main tool is the Barge-Sullivan rational surgery realization theorem (\cite{Ba}, \cite{Sul}, \cite{Mil}) as well as the Hattori-Stong theorems.
Among other things, they showed that an RPP exists in dimension $n\leq512$ if and only if $n\in\{4,8,16,32,128,256\}$ (\cite[Thm 1.1]{Su1}, \cite[Thm A]{KS}).
Moreover, whether a given dimension $n=8(2^a+2^b)$ supports
an RPP is reduced to a solvability of a quadratic residue equation (\cite[Thms 6,9]{KS}).

Recall that the Betti numbers $b_i=b_i(M)$ of a \emph{general} $n$-dimensional (closed connected orientable) manifold $M$ satisfy
\be\label{Betti restrictions} b_0=b_n=1,\qquad b_i=b_{n-i},\qquad \text{$b_{n/2}$ is even whenever $n=4k-2$}.\ee
The first two restrictions in (\ref{Betti restrictions}) are clear and the third one is due to the fact that in this case the intersection pairing on $H^{2k-1}(M;\mathbb{R})$ is both skew-symmetric and non-degenerate. Conversely, given a sequence of nonnegative integers $(b_0,b_1,\ldots,b_n)$ satisfying (\ref{Betti restrictions}), can we find an $n$-dimensional manifold whose Betti numbers realize this sequence? The following proposition tells us that such realization problem is deeply related to the existence of an RPP when $n=4k$.
\begin{proposition}
Let $(b_0,b_1,\ldots,b_n)$ be a sequence of nonnegative integers satisfying the restrictions in (\ref{Betti restrictions}).
\begin{enumerate}
\item
When $n$ is odd, $n=4k-2$, or $n=4k$ and $b_{2k}$ is even, there exists an $n$-dimensional manifold whose Betti numbers realize this sequence.

\item
When $n=4k$ and $b_{2k}$ is odd, there exists an $n$-dimensional manifold whose Betti numbers realize this sequence if there exists a $4k$-dimensional RPP.
\end{enumerate}
\end{proposition}
\begin{remark}
\begin{enumerate}
\item
Proposition \ref{Betti restrictions} is known to some experts. For instance, we were told by Yang Su that M. Kreck knows it many years ago. But it seems that it was never written down in the literature, to our best knowledge. So we present it here for the reader's convenience.

\item
As discussed above, in some dimensions RPP are known to exist, so in these dimensions the answer to the Betti number realization problem are also affirmative. For some dimension $4k$, if an RPP does not exist in this dimension, we do not know how to construct a desired manifold realizing a general prescribed Betti numbers.
\end{enumerate}
\end{remark}
\begin{proof}
The basic idea is to apply connected sum operations to the building blocks $S^i\times S^{n-i}$ ($1\leq i\leq\lfloor n/2\rfloor$).

When $n$ is odd, we may take connected sums of $b_i$ copies of $S^i\times S^{n-i}$ for all $1\leq i\leq\frac{n-1}{2}$.

When $n=4k-2$ (resp. $n=4k$ and $b_{2k}$ is even), we may take connected sums of $b_i$ copies of $S^i\times S^{n-i}$ for all $1\leq i\leq 2k-2$ (resp. $1\leq i\leq 2k-1$) and $b_{2k-1}/2$ (resp. $b_{2k}/2$) copies of $S^{2k-1}\times S^{2k-1}$ (resp. $S^{2k}\times S^{2k}$).

When $n=4k$ and $b_{2k}$ is odd. We may take connected sums of $b_i$ copies of $S^i\times S^{n-i}$ for all $1\leq i\leq b_{2k-1}$, $\frac{b_{2k}-1}{2}$ copies of $S^{2k}\times S^{2k}$, and an RPP.
\end{proof}

\section*{Acknowledgements}
The authors thank the referee for some useful comments, which make this paper more readable.

\end{document}